\documentclass[12pt,leqno]{article}

\usepackage{amsfonts,amsthm,amssymb,amsmath,graphicx}

\usepackage{color}

\newcommand{\fc}{\mathfrak {c}}
\newcommand{\fe}{\mathfrak {e}}

\newcommand{\fg}{\mathfrak {g}}
\newcommand{\fh}{\mathfrak {t}}

\newcommand{\ft}{\mathfrak {t}}
\newcommand{\fa}{\mathfrak {a}}
\newcommand{\fk}{\mathfrak {k}}
\newcommand{\fx}{\mathfrak {h}}
\newcommand{\fkc}{\mathfrak {k}^{\mathbb {C}}}
\newcommand{\fl}{\mathfrak {l}}
\newcommand{\fp}{\mathfrak {p}}

\newcommand{\fpc}{\mathfrak {p}^{\mathbb {C}}}
\newcommand{\fo}{\mathfrak {o}}
\newcommand{\fs}{\mathfrak {s}}
\newcommand{\fu}{\mathfrak {u}}
\newcommand{\fz}{\mathfrak {z}}

\newcommand{\fgc}{\mathfrak {g}^{\mathbb {C}}}

\newcommand{\si}{\sigma}

\newcommand{\thq}{\theta}

\newcommand{\sk}{\bigskip}

\newcommand{\lra}{\longrightarrow}

\newcommand{\pa}{{\rm P}}

\newcommand{\dy}{{\rm D}}

\newcommand{\co}{{\cal O}}

\newcommand{\De}{\Delta}

\newcommand{\Ga}{\Gamma}

\newcommand{\bsi}{\widehat{\sigma}}
\newcommand{\bth}{\widehat{\theta}}

\newcommand{\al}{\alpha}
\newcommand{\be}{\beta}

\newcommand{\ga}{\gamma}

\newcommand{\bsl}{\backslash}

\newcommand{\aut}{{\rm Aut}}

\newcommand{\inn}{{\rm Int}}

\newcommand{\Ad}{{\rm Ad}}
\newcommand{\ad}{{\rm ad}}

\newtheorem{proposition}{Proposition}[section]
\newtheorem{theorem}[proposition]{Theorem}
\newtheorem{corollary}[proposition]{Corollary}
\newtheorem{lemma}[proposition]{Lemma}

\newcommand{\bk}{\mathbb {K}}
\newcommand{\bz}{\mathbb {Z}}
\newcommand{\bc}{\mathbb {C}}
\newcommand{\bh}{\mathbb {H}}

\newcommand{\br}{\mathbb {R}}

\newcommand{\bct}{\mathbb {C}^\times}

\begin{document}

\begin{center}
{\large \bf Outer Automorphism Groups of
Simple Lie Algebras \\
and Symmetries of Painted Diagrams}
\end{center}

\sk

Meng-Kiat Chuah

Department of Mathematics, National Tsing Hua University,

Hsinchu 300, Taiwan.

chuah@math.nthu.edu.tw

\sk

Mingjing Zhang

Department of Mathematics, National Tsing Hua University,

Hsinchu 300, Taiwan.

mjzhang@math.nthu.edu.tw

\sk

\sk

\sk

\sk

\sk

\noindent {\bf Abstract: }

Let $\fg$ be a simple Lie algebra.
Let $\aut(\fg)$ be the group of all automorphisms on $\fg$,
and let $\inn(\fg)$ be its identity component.
The outer automorphism group of $\fg$
is defined as $\aut(\fg)/\inn(\fg)$.
If $\fg$ is complex and has Dynkin diagram $\dy$,
then $\aut(\fg)/\inn(\fg)$ is isomorphic to $\aut(\dy)$.
We provide an analogous result for the real case.
For $\fg$ real, we let $\fg$ be represented by
a painted diagram $\pa$. Depending on whether
the Cartan involution of $\fg$
belongs to $\inn(\fg)$, we show that
$\aut(\fg)/\inn(\fg)$ is isomorphic to
$\aut(\pa)$ or $\aut(\pa) \times \bz_2$.
This result extends to the outer automorphism groups
of all real semisimple
Lie algebras.

\sk

\sk

\noindent {\bf 2010 Mathematics Subject Classification:}

17B20, 17B40.

\sk

\sk

\noindent {\bf Keywords: }

outer automorphism group, simple Lie algebra,
Dynkin diagram, painted diagram.

\sk

\sk


\newpage

\section{Introduction} \label{intro}
\setcounter{equation}{0}

Let $\fg$ be a finite dimensional simple Lie algebra over
$\bk \in \{\bc, \br\}$.
Let $\aut(\fg)$ be the group of all automorphisms on $\fg$,
and let $\inn(\fg)$ be its identity component. The group
\[ \aut(\fg)/\inn(\fg) \]
is called the {\it outer automorphism group} of $\fg$.

Suppose that $\bk = \bc$. A Cartan subalgebra $\fh$ of $\fg$
leads to the root space decomposition $\fg = \fh + \sum_\De \fg_\al$.
The Dynkin diagram $\dy$ of $\fg$ represents a simple system
$\Pi \subset \De$. Let $\aut(\dy)$ denote its group of
diagram automorphisms.
Every $[\si] \in \aut(\fg)/\inn(\fg)$ has a representative $\si$
of finite order, and there exist corresponding
$\si$-stable $\fh$ and $\si$-stable $\Pi$
(this means $\si$ stabilizes $\sum_\Pi \fg_\al$).
Since $\fh$ is $\si$-stable, we obtain an action $\bsi$ on $\De$
by $\si \fg_\al = \fg_{\bsi \al}$.
The restriction of $\bsi$ to $\Pi$ leads to $\bsi \in \aut(\dy)$.
The following is a classical result;
see for example \cite[Thm.7.8]{kn}.

\begin{theorem}
Let $\fg$ be a complex simple Lie algebra.
We have a natural isomorphism
\[ \aut(\fg)/\inn(\fg) \cong \aut(\dy) \;,\;
[\si] \mapsto \bsi .\]
\label{class}
\end{theorem}

Now let $\bk = \br$.
We assume that $\fg$ is a noncompact real form of a
complex simple Lie algebra $\fg^\bc$.
The outer automorphism groups were initially
studied in \cite{mu},
and more recently computed
on a case-by-case basis
and listed in a table for all $\fg$ \cite[Cor.2.15]{gu}.
This article adds more insight to these groups
by a uniform treatment in the spirit of Theorem \ref{class}.
It provides their realizations as
diagram automorphisms, together with the explicit isomorphisms.
Our result also extends to the outer automorphism groups
of all real semisimple Lie algebras.

Just as Dynkin diagrams represent complex simple Lie algebras,
we shall let painted diagrams represent real simple
Lie algebras $\fg$.
Let $\thq$ be a Cartan involution of $\fg$, and let
$\fg = \fk + \fp$ be its Cartan decomposition.
We always regard $\fpc$ as a $\fkc$-module by
the adjoint $\fkc$-representation $[\fkc,\fpc] \subset \fpc$.
If $\fh$ is a Cartan subalgebra of $\fk$, we write
$\fg^\bc = \fh^\bc + \sum_\De \fg_\al^\bc$,
where $\De \subset (\ft^\bc)^*$ are the restricted roots.
So $\dim \fg_\al^\bc = 1,2$.
If $\Pi_\fk \subset \De$ is a simple system of $\fk$,
we let $\Pi_\fp$ be the corresponding lowest roots
of the $\fkc$-module $\fpc$.
A {\it painted diagram} is a diagram $\pa$
whose vertices represent some roots in $\De$ with the usual edge
relation of Dynkin diagram,
and each vertex is of white or black color.
We say that $\pa$ represents $\fg$ if
there exist $\Pi_\fk \cup \Pi_\fp$ such that:
\begin{equation}
\begin{array}{cl}
\mbox{(a)} &
\mbox{white vertices represent $\Pi_\fk$,}\\
\mbox{(b)} &
\mbox{black vertices represent $\Pi_\fp$.}
\end{array}
\label{ree}
\end{equation}

There are unique positive integers $\{a_\al \;;\; \al \in \pa\}$
without nontrivial common factor such that $\sum_\pa a_\al \al = 0$.
These $a_\al$ are given in
Tables 1-2 of \cite[Ch.X-\S5]{he}.
By a theorem of Kac,
the painted diagrams which represent real simple Lie algebras
are obtained by imposing the condition
\[ r \sum_{\rm black} a_\al =2 \]
in Table $r$, for $r=1,2$ \cite[Thm.2.1,Cor.2.2]{tams}.
The painted diagrams appear in
\cite[Table 7]{ov}\cite[Figs.1-3]{tams} for all $\fg$.
For the reader's convenience, we reproduce
them in Section 2.

Let $\aut(\pa)$ denote diagram automorphisms which
preserve vertex colors.
Consider $\si \in\aut(\fg)$ of finite order
and commutes with $\thq$.
We extend it to a $\bc$-linear automorphism on $\fgc$.
There exists a $\si$-stable Cartan subalgebra $\fh$ of $\fk$
which has a $\si$-stable simple system $\Pi_\fk$.
Then the corresponding lowest roots $\Pi_\fp$ of $\fpc$
are also $\si$-stable.
By restricting $\si$ to the root spaces of $\Pi_\fk \cup \Pi_\fp$,
we obtain $\bsi \in \aut(\pa)$ by
$\si \fg_\al^\bc = \fg_{\bsi \al}^\bc$.
There exist $\{c_\al \in \bct \;;\; \al \in \pa\}$
and root vectors $\{X_\al \;;\; \al \in \pa\}$ such that
\begin{equation}
 \{X_\al \in \fg_\al^\bc \cap \fk^\bc \;;\;
\mbox{white } \al \in \pa\} \cup
\{X_\al \in \fg_\al^\bc \cap \fp^\bc \;;\;
\mbox{black } \al \in \pa\}
\;,\; \si X_\al = c_{\bsi \al} X_{\bsi \al} .
\label{repr}
\end{equation}
Here $\bct$ denotes the nonzero complex numbers.
We shall see in Lemma \ref{kim} that in fact $|c_\al| =1$.

The Cartan involution $\thq$ extends to a $\bc$-linear involution
on $\fg^\bc$. By Theorem \ref{class},
it leads to $\bth \in \aut(\dy)$.
We have \cite[Ch.IX,Thm.5.7]{he}
\begin{equation}
\begin{array}{c}
\mbox{order } \bth = 1 \; \Longleftrightarrow \;
\mbox{rank } \fg = \mbox{rank } \fk \; \Longleftrightarrow \;
\thq \in \inn(\fg) , \\
\mbox{order } \bth = 2 \; \Longleftrightarrow \;
\mbox{rank } \fg > \mbox{rank } \fk \; \Longleftrightarrow \;
\thq \not\in \inn(\fg) .
\end{array}
\label{ang}
\end{equation}

Let $r = \mbox{order } \bth$, and let $\bz_r = \bz/r \bz$.
For convenience, we write them as multiplicative groups
$\bz_1 \cong \{1\}$ and $\bz_2 \cong \{\pm 1\}$.
We now present the main result of this article.

\begin{theorem}
Let $\fg$ be a noncompact real form of a complex simple Lie algebra.
We have a natural isomorphism
\begin{equation}
 \aut(\fg)/\inn(\fg) \cong \aut(\pa) \times \bz_r \;,\;
[\si] \mapsto
\left\{
\begin{array}{cl}
(\bsi, 1) & \mbox{if } r=1 ,\\
(\bsi, \prod_\pa c_\al^{a_\al}) & \mbox{if } r=2.
\end{array}
\right.
\label{them}
\end{equation}
\label{main}
\end{theorem}

A real simple Lie algebra is either a complex simple
Lie algebra regarded as real, or a real form of a complex simple
Lie algebra \cite[Thm.6.94]{kn}.
So a real semisimple Lie algebra $\fg$ is of the form
\[ \fg = \fc_1 + ...  + \fc_m + \fg_1 + ... + \fg_n ,\]
where $\fc_i$ and $\fg_j$ are simple ideals,
each $\fc_i$ is complex,
and each $\fg_j$ is a real form of a complex simple Lie algebra.
Let $\dy = \dy_1 \cup ... \cup \dy_m$
be the Dynkin diagram of $\fc_1 + ...  + \fc_m$,
and let $\pa = \pa_1 \cup ... \cup \pa_n$ be the painted diagram
of $\fg_1 + ... + \fg_n$. If $\fg_i$ is compact,
then $\pa_i$ is just the Dynkin diagram of
$\fg_i^\bc$ without black vertex.
Let $\thq_i$ be a Cartan involution of $\fg_i$, and
let $[\thq_i]$ denote its quotient in $\aut(\fg)/\inn(\fg)$,
as discussed in (\ref{ang}).
Then Theorem \ref{main} leads to the following.

\begin{corollary}
Using the above notations,
let $\fg = \fc_1 + ...  + \fc_m + \fg_1 + ... + \fg_n$
be a real semisimple Lie algebra. Then
\[ \aut(\fg)/\inn(\fg) \cong
(\aut(\dy) \times \bz_2^m) \times
(\aut(\pa) \times \langle [\thq_1], ..., [\thq_n] \rangle) . \]
\label{exte}
\end{corollary}

Here $\bz_2^m$ is generated by a Cartan involution of
$\fc_1 + ...  + \fc_m$, which is conjugation with respect to
a maximally compact subalgebra.

In Section 2, we list the painted diagrams
of all real simple Lie algebras.
In Section 3, we prove Theorem \ref{main}
and Corollary \ref{exte}.

\sk

\noindent {\bf Acknowledgements }

This work is partially supported
by a research grant from the Ministry of Science and Technology
of Taiwan.
The reviewers provide valuable comments to improve the contents
and presentations of this article.


\newpage

\section{Painted Diagrams}
\setcounter{equation}{0}

For convenience of reference, we list
the painted diagrams of all real simple
Lie algebras $\fg$, reproduced from
\cite[Table 7]{ov}\cite[Figs.1-3]{tams}.
They represent $\fg$ by (\ref{ree}).

We use Cartan's notation
for the real forms of exceptional Lie algebras
\cite[Ch.10-\S6]{he}, for example
$\fe_{6(-14)}$ is the real form of $\fe_6$
with $\dim \fp - \dim \fk = -14$.

For the classical real forms, the sizes of the painted diagrams
in Figures 1, 2 and 3
are determined by the number of white vertices as follows:

$\fs \fu(p,q): (p-1,q-1)$,

$\fs \fo(2,2q+1), \fs \fo(2,2q): q$,

$\fs \fp(n,\br), \fs \fo^*(2n): n-1$,

$\fs \fo(2p,2q+1), \fs \fp(p,q), \fs \fo(2p,2q),
\fs \fo(2p+1,2q+1): (p,q)$,

$\fs \fl(2p+1,\br), \fs \fl(p,\bh), \fs \fl(2p,\br): p$.

\noindent For example in Figure 1, the painted diagram of $\fs \fu(p,q)$
has $p-1$ (resp. $q-1$) white vertices in the left (resp. right)
connected component, and so on.
The generic diagram of $\fs \fu(p,q)$ requires
$p> 1$ or $q>1$.

We regard $\fpc$ as $\fkc$-modules by the adjoint
representation $[\fkc,\fpc] \subset \fpc$.

$\;$

\sk

\begin{figure}[h]
{\qquad \includegraphics[width=13.2cm,height=4.8cm]{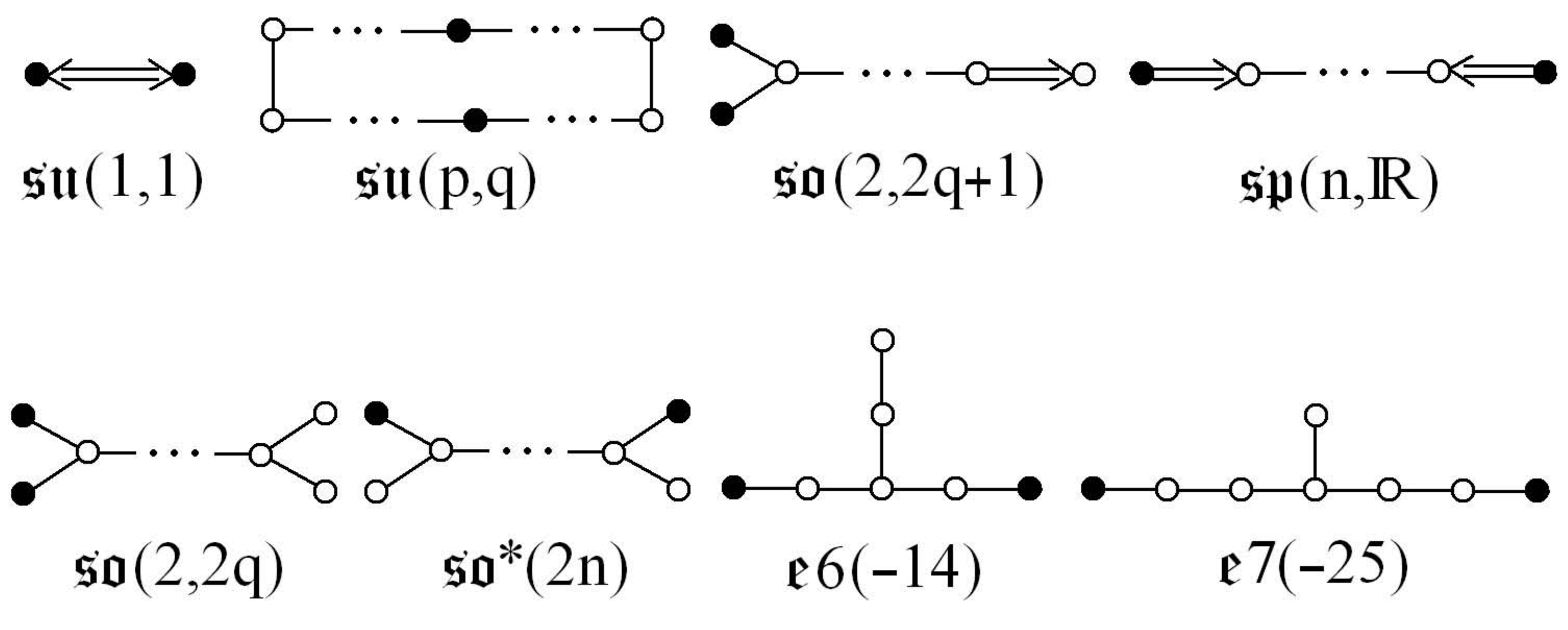}}
\caption{rank $\fg$ = rank $\fk$,
and $\fpc$ is the sum of two irreducible factors.}
\end{figure}

$\;$

\sk

\begin{figure}[h]
{\qquad \includegraphics[width=13.2cm,height=7.2cm]{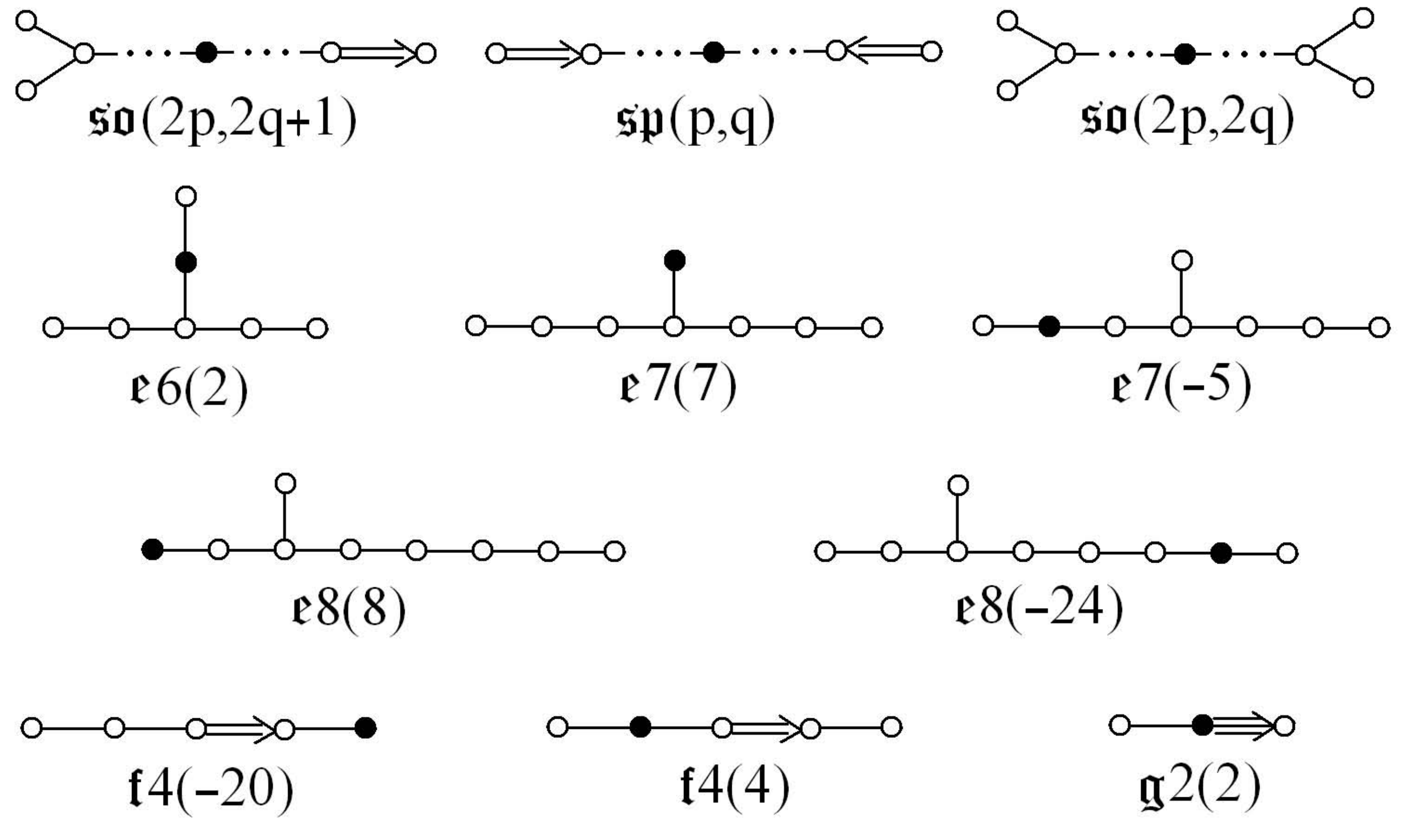}}
\caption{rank $\fg$ = rank $\fk$,
and $\fpc$ is irreducible.}
\end{figure}

$\;$

\sk

$\;$

$\;$

\sk

$\;$

$\;$

\sk

$\;$

$\;$

\sk

$\;$

\begin{figure}[h]
{\qquad \includegraphics[width=12.4cm,height=3.1cm]{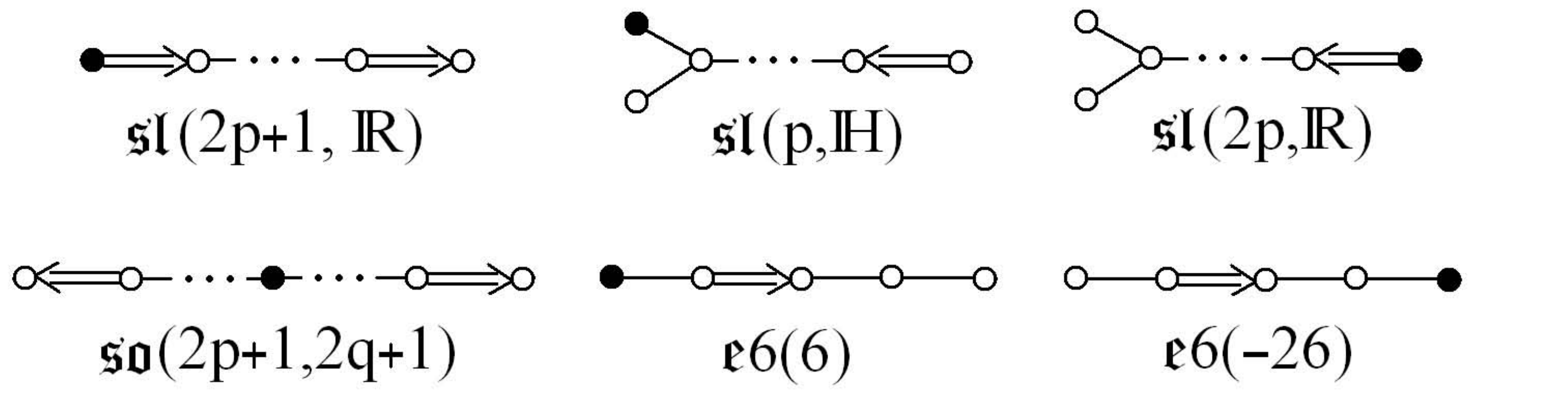}}
\caption{rank $\fg$ $>$ rank $\fk$,
and $\fpc$ is irreducible.}
\end{figure}

$\;$

$\;$

\sk

$\;$

$\;$

\sk

$\;$


\newpage

\section{Outer Automorphism Groups}
\setcounter{equation}{0}

In this section, we prove Theorem \ref{main} and Corollary \ref{exte}.
We first review the structures of
real simple Lie algebras and painted diagrams \cite{tams}.
Let $\fg$ be a noncompact real form of a complex simple Lie algebra $\fgc$.
Let $\thq$ be a Cartan involution, with
Cartan decomposition $\fg = \fk + \fp$.
Let $\aut(\fg)$ be the group of all automorphisms on $\fg$,
and let $\inn(\fg)$ be its identity component.

\begin{theorem}
{\rm \cite[Ch.IX,Thm.5.7]{he}}
We have $\thq \in \inn(\fg)$ if and only if
$\mbox{rank } \fg = \mbox{rank } \fk$.
\label{helg}
\end{theorem}

Let $\fh$ be a Cartan subalgebra of $\fk$, with root system $\De$.
We have the root space decomposition
$\fgc = \ft^\bc + \sum_\De \fg_\al^\bc$.
If rank $\fg$ = rank $\fk$,
then $\dim \fg_\al^\bc =1$ for all $\al \in \De$.
If rank $\fg$ $>$ rank $\fk$, then either $\dim \fg_\al^\bc =1$
or $\dim (\fg_\al^\bc \cap \fkc) = \dim (\fg_\al^\bc \cap \fpc) =1$.

Let $\fg_\al = \fg \cap (\fg_\al^\bc + \fg_{-\al}^\bc)$, and
we have $\fg_\al = \fg_{-\al}$.
If $\De^+ \subset \De$ is a positive system, then
\[ \fg = \ft + \sum_{\De^+} \fg_\al .\]

Let $\fkc$ act on $\fpc$ by $[\fkc,\fpc] \subset \fpc$.
Each $\fg$ belongs to one of the following cases.
\begin{equation}
\begin{array}{cl}
\mbox{(a)} & \mbox{rank } \fg = \mbox{rank } \fk,
\mbox{and $\fpc$ is the sum of two irreducible factors.}\\
\mbox{(b)} & \mbox{rank } \fg = \mbox{rank } \fk,
\mbox{and $\fpc$ is irreducible.}\\
\mbox{(c)} & \mbox{rank } \fg > \mbox{rank } \fk,
\mbox{and $\fpc$ is irreducible.}
\end{array}
\label{tig}
\end{equation}
Write $\fk = [\fk, \fk] + \fz$, where
the commutator subalgebra $[\fk,\fk]$ is semisimple,
and $\fz$ is the center of $\fk$.
We have $\dim \fz = 1$ in (\ref{tig})(a),
and $\fz =0$ in (\ref{tig})(b,c).

A painted diagram is a diagram $\pa$
whose vertices are white or black, and they represent a subset
of $\De$ with the usual edge relation of Dynkin diagram.
We say that $\pa$ represents $\fg$ if (\ref{ree}) holds.
The painted diagrams in Figures 1, 2 and 3 represent
cases (\ref{tig})(a,b,c) respectively.

Let $\aut(\pa)$ denote automorphisms on $\pa$ which preserve vertex colors.
If $d \in \aut(\pa)$, then a $d$-orbit $\co \subset \pa$
can consist of 1, 2, 3 or 4 elements.
For example if $d$ fixes $\al$, then $\co = \{\al\}$ has one element.
If $\fg = \fs \fo(4,4)$, then Figure 2 shows that
$\co$ may have 4 elements.

Since $\pa$ represents $\fg$ (or equivalently $\thq$),
additional data on $\pa$ represent members of $\aut(\fg)$.
Let $S^1 \subset \bct$ be the unit circle.
We shall let
$d \in \aut(\pa)$ and $\{c_\al \in S^1\}_\pa$
 represent $\si \in \aut(\fg)$ by (\ref{repr}), where
$\bsi =d$. Here $\si$ is identified with its $\bc$-linear extension
to $\aut(\fgc)$.

\begin{lemma}
$\;$

\noindent {\rm (a)} $\,$ Suppose that $\dim \fg_\al^\bc =1$.
A vector space automorphism $\si$ on $\fg_\al^\bc + \fg_{-\al}^\bc$
stabilizes $\fg_{\pm \al}^\bc$ and $\fg_\al$
if and only if there exists $z \in S^1$ such that
$\si$ has eigenvalue $z^{\pm 1}$ on $\fg_{\pm \al}^\bc$.

\noindent {\rm (b)} $\,$ If $\dim \fg_\al^\bc =2$,
we get the same result by replacing $\fg_{\pm \al}^\bc$ and $\fg_\al$
of part (a) with $\fg_{\pm \al}^\bc \cap \fkc$ and $\fg_\al \cap \fk$,
or $\fg_{\pm \al}^\bc \cap \fpc$ and $\fg_\al \cap \fp$.

\noindent {\rm (c)} $\,$ Let $\si \in \aut(\fg)$.
If $\bsi \in \aut(\pa)$
and $\{c_\al \in \bct\}_\pa$ satisfy (\ref{repr}),
then $\prod_\co c_\al \in S^1$
for all $\bsi$-orbit $\co \subset \pa$.
We may replace with $\{c_\al' \in S^1\}_\pa$
and (\ref{repr}) remains satisfied.
\label{kim}
\end{lemma}
\begin{proof}
Suppose that $\dim \fg_\al^\bc =1$.
Let $G$ be a connected Lie group whose Lie algebra is $\fg$,
and let $T$ be the connected subgroup of $G$ whose Lie algebra
is $\ft$. Let $\exp$ be the exponential maps $\ft \lra T$
and $i \br \lra S^1$.
The root $\al$ is integral, namely
there exists a group homomorphism
$\chi : T \lra S^1$ such that
$\chi(\exp X) = \exp(\al(X))$ for all $X \in \ft$.

There exist $X_{\pm \al} \in \fg_{\pm \al}^\bc$
such that $X_\al + X_{-\al} \in \fg_\al$.
For all $t \in T$, $\Ad_t$ stabilizes $\fg_\al$.
So $\Ad_t (X_\al + X_{-\al}) =
\chi(t) X_\al + (\chi(t))^{-1} X_{-\al} \in \fg_\al$.
Therefore,
\begin{equation}
 \fg_\al = \{ e^{ix} X_\al + e^{-ix} X_{-\al} \;;\;
x \in \br \} .
\label{wese}
\end{equation}

Let $\si$ be a vector space automorphism on
$\fg_\al^\bc + \fg_{-\al}^\bc$.
Suppose that $\si$ has eigenvalues $z$ and $w$ respectively
on $\fg_\al^\bc$ and $\fg_{-\al}^\bc$.
If $\si$ stabilizes $\fg_\al$, then (\ref{wese})
implies that $z = w^{-1} \in S^1$.
Conversely, if $\si$ has eigenvalues $z^{\pm 1} \in S^1$
on $\fg_{\pm \al}^\bc$, then (\ref{wese}) implies that
$\si$ stabilizes $\fg_\al$. This proves Lemma \ref{kim}(a).
The proof of Lemma \ref{kim}(b) is similar.

Finally we prove Lemma \ref{kim}(c).
Let $\si \in \aut(\fg)$.
Suppose that $d \in \aut(\pa)$ and
$\{c_\al \in \bct\}_\pa$ satisfy (\ref{repr})
for $\bsi = d$.
Let $\co$ be a $d$-orbit, and let $m$ be
its number of elements.
If $\be \in \co$, and $X_\be$ is as given in (\ref{repr}), then
\begin{equation}
\si^m X_\be = (\prod_{\al \in \co} c_\al) X_\be .
\label{jet}
\end{equation}
So $\si^m$ stabilizes $\bc X_\be$.
By applying Lemma \ref{kim}(a,b) to $\si^m$, we get
$\prod_{\al \in \co} c_\al \in S^1$.

We use the polar coordinates
$c_\al = r_\al \exp{i \thq_\al} \in \br^+ \times S^1$,
and let $c_\al' = \exp{i \thq_\al} \in S^1$.
Then $\prod_{\al \in \co} c_\al =
\prod_{\al \in \co} c_\al'$ for each $d$-orbit $\co$, and
$(c',d)$ is again related to $\si$ by (\ref{repr})
\cite[Thm.1.4,(5.9)(b)]{tams}.
This proves the lemma.
\end{proof}

\sk

The above lemma leads to the notion of markings on $\pa$
as a tool to study $\aut(\fg)$.
A {\it marking} on $\pa$ is a pair $(c,d)$, where
\begin{equation}
 d \in \aut(\pa) \;,\; c_\al \in S^1
\mbox{ for all } \al \in \pa .
\label{aat}
\end{equation}
We say that $(c,d)$ represents $\si$ if
there exist root vectors $\{X_\al\}_\pa$
which satisfy (\ref{repr}) for $\bsi = d$.
Some examples of markings appear in \cite[\S9]{tams}.
Lemma \ref{kim} explains why we have $c_\al \in S^1$
in (\ref{aat}), instead of
merely $c_\al \in \bct$.

Let $\aut_n(\fg)$ denote $\fg$-automorphisms of order $n$.
Given $\si \in \aut_n(\fg)$, there exists a Cartan involution
$\thq$ which commutes with $\si$, and
they both extend to $\fgc$.
This identifies $\aut_n(\fg)$ with commuting
$\aut_2(\fgc) \times \aut_n(\fgc)$ \cite[Cor.2.4]{tams}.
Every member of $\aut_n(\fg)$ can be represented
by a marking on $\pa$ \cite[Thm.1.3]{tams}.
Conversely, a marking represents a $\fg$-automorphism
if it satisfies an {\it admissible} condition
\cite[(1.5)]{tams}.
For example, suppose that $\mbox{rank } \fg = \mbox{rank } \fk$.
Let $\be \in \pa$ satisfy $a_\be =1$, where
$\{a_\al\}_\pa$ are given in Table 1 of \cite[Ch.X-\S5]{he}.
Then $\be$ is the
lowest root with respect to the simple system represented
by the Dynkin diagram $\pa \bsl \{\be\}$ of $\fgc$.
By Lemma \ref{kim},
we can construct $\si \in \aut(\fg)$ by assigning arbitrary
eigenvalues $c_\al \in S^1$ to $\fg_\al^\bc$ for all
$\al \in \pa \bsl \{\be\}$.
Then the highest root space $\fg_{-\be}^\bc$ has eigenvalue
$\prod_{\pa \bsl \{\be\}} c_\al^{a_\al}$, so
$\fg_\be^\bc$ has eigenvalue
$c_\be = (\prod_{\pa \bsl \{\be\}} c_\al^{a_\al})^{-1}$.
Hence the admissible condition for the case of
$\mbox{rank } \fg = \mbox{rank } \fk$ with $d=1$ is
\begin{equation}
 \prod_{\pa} c_\al^{a_\al} = 1 .
 \label{adc}
 \end{equation}
Slight modifications are needed for the admissible condition when
$\mbox{rank } \fg > \mbox{rank } \fk$ or $d \neq 1$.

\begin{proposition}
Let $\si \in \aut_n(\fg)$ be represented by a marking $(c,d)$
on $\pa$ with $d=1$.

\noindent {\rm (a)} $\,$ If $\mbox{rank } \fg = \mbox{rank } \fk$,
then $\si \in \inn(\fg)$.

\noindent {\rm (b)} $\,$ If $\mbox{rank } \fg > \mbox{rank } \fk$
and $\prod_\pa c_\al^{a_\al} = 1$, then $\si \in \inn(\fg)$.
\label{satu}
\end{proposition}
\begin{proof}
Let $\si \in \aut_n(\fg)$ be represented by a marking $(c,d)$
on $\pa$ with $d=1$.
The white vertices form the Dynkin diagram of $[\fk,\fk]$,
and $d=1$ on the white vertices.
So by Theorem \ref{class}, $\si|_{[\fk,\fk]} \in \inn([\fk,\fk])$,
namely $\si|_{[\fk,\fk]} = \prod_i \exp (\ad_{X_i})$
for some $X_i \in [\fk,\fk]$.

We regard $X_i$ as elements of $\fk$ and get
$\prod_i \exp (\ad_{X_i}) \in \inn(\fk)$.
It acts trivially on $\fz$ because
$\ad_{X_i}$ annihilates $\fz$. We have $d =1$ on the black vertices,
so $\si$ also acts trivially on $\fz$ \cite[Prop.4.5(a)]{tams}.
Therefore, $\si|_\fk = \prod_i \exp (\ad_{X_i})$.
Since $X_i$ are also elements of $\fg$, the expression
$\prod_i \exp (\ad_{X_i})$
extends to some $\tau \in \inn(\fg)$. Namely,
\begin{equation}
 \si|_\fk = \tau|_\fk \;;\; \tau \in \inn(\fg) .
 \label{geat}
 \end{equation}

 We now prove Proposition \ref{satu}(a), so
 suppose that $\mbox{rank } \fg = \mbox{rank } \fk$.
 There are two cases in (\ref{tig})(a,b), and
 we first consider (\ref{tig})(a).
By (\ref{geat}), there exists $\mu \in \aut(\fg)$ such that
\[ \si = \tau \mu \;,\;
\mu|_\fk = 1_\fk ,\]
where $1_\fk$ is the identity mapping on $\fk$.
Let $\be, \ga$ be the two black vertices of $\pa$.
They satisfy $a_\be = a_\ga =1$,
and they represent the lowest roots of $\fpc$.
There exists $z \in S^1$ such that $\mu$
has eigenvalues $z$ and $z^{-1}$ on
$\fg_\be^\bc$ and $\fg_\ga^\bc$ respectively.

Let $f: [0,1] \lra S^1$ be a path such that
$f(0)=1$ and $f(1) = z$.
For each $0 \leq t \leq 1$, we define a marking
$(c_t, d)$ on $\pa$ by $d=1$ and
\[ (c_t)_\be = f(t) \;,\; (c_t)_\ga = f(t)^{-1} \;,\;
(c_t)_\al = 1 \mbox{ for all $\al \in \pa \bsl \{\be,\ga\}$}.\]
By Lemma \ref{kim} and (\ref{adc}),
for all $t$, $(c_t,d)$ represents some $\mu_t \in \aut(\fg)$.
We have $\mu_0 = 1$ and $\mu_1 = \mu$.
So $\mu \in \inn(\fg)$ and
hence $\si \in \inn(\fg)$.

Next we consider the case (\ref{tig})(b).
There are two extensions of $\si|_\fk$
to $\aut(\fg)$ \cite[Cor.1.6]{tams}.
By (\ref{geat}), one of them is $\tau \in \inn(\fg)$,
so the other is $\tau \thq$. By Theorem \ref{helg},
we have $\thq \in \inn(\fg)$, so $\tau \thq \in \inn(\fg)$.
Since $\si$ is either $\tau$ or $\tau \thq$,
it follows that $\si \in \inn(\fg)$.
This proves Proposition \ref{satu}(a).

We now prove Proposition \ref{satu}(b),
so suppose that $\mbox{rank } \fg > \mbox{rank } \fk$.
There are two extensions of $\si|_\fk$ to $\aut(\fg)$
\cite[Cor.1.6]{tams}, and they are
$\si, \si \thq \in \aut(\fg)$.
By (\ref{geat}), one of $\{\si, \si \thq\}$ belongs to $\inn(\fg)$.
The other of $\{\si, \si \thq\}$ does not belong to $\inn(\fg)$
because $\thq \not\in \inn(\fg)$ by Theorem \ref{helg}.
We conclude that
\begin{equation}
 \sharp (\{\si, \si \thq\} \cap \inn(\fg)) = 1 .
 \label{ttn}
 \end{equation}

Suppose in addition that $\prod_\pa c_\al^{a_\al} = 1$.
Let $\be \in \pa$ be the unique black vertex.
Let $\si \thq$ be represented by the marking $(c',1)$.
Then $c_\al = c_\al'$ for all white vertices $\al$,
and $c_\be' = - c_\be$. Hence $\prod_\pa (c_\al')^{a_\al} = -1$.
Extend $\si$ and $\si \thq$ to $\bc$-linear involutions
on $\fg^\bc$.
By \cite[Prop.8.2]{tams},
\begin{equation}
 \si \in \inn(\fg^\bc) \;,\; \si \thq \not\in \inn(\fg^\bc) .
 \label{tts}
 \end{equation}

 The elements of $\inn(\fg)$ extend naturally to $\inn(\fg^\bc)$.
 So (\ref{ttn}) and (\ref{tts}) imply that $\si \in \inn(\fg)$.
 This proves Proposition \ref{satu}(b).
\end{proof}

\sk

While Proposition \ref{satu} provides a sufficient condition
for $\si \in \inn(\fg)$, the next proposition provides
a necessary condition.

\begin{proposition}
If $\si \in \aut_n(\fg)$ is represented by a marking $(c,d)$
on $\pa$ with $d \neq 1$,
then $\si \not\in \inn(\fg)$.
\label{onto}
\end{proposition}
\begin{proof}
Let $G$ be a connected Lie group whose Lie algebra is $\fg$,
and let $K$ be its subgroup whose Lie algebra is $\fk$.
We have the Cartan decomposition $G = K \cdot \exp(\fp)$.
Let $\fx = \ft + \fa$ be a maximally compact Cartan subalgebra
of $\fg$,
where $\ft = \fx \cap \fk$ and $\fa = \fx \cap \fp$.
So $\ft$ is a Cartan subalgebra of $\fk$.
Let $N$ denote the normalizer with respect to the adjoint action,
and let $H = N_G(\fx)$ be the normalizer of $\fx$ in $G$.
Write $H = TA \subset K \cdot \exp(\fp)$,
where $T = N_K(\fx)$,
and $A$ is diffeomorphic to $\fa$.

Let $\si \in \aut_n(\fg)$ be represented by a marking $(c,d)$
on $\pa$ with $d \neq 1$.
Assume that $\si \in \inn(\fg)$, namely $\si = \Ad_g$,
and we seek a contradiction.
Since $\si$ stabilizes $\fx$, we have $g = ta \in TA$.
Here $\si$ stabilizes $\fk$, hence
\begin{equation}
\si = \Ad_t \;,\; t \in T .
\label{bak}
\end{equation}
There are two cases to consider.

{\it Case 1: $d \neq 1$ on the white vertices.}

Condition (\ref{bak}) implies that
$\si|_\fk \in \inn(\fk)$. By Theorem \ref{class}, this is
a contradiction because $d \neq 1$ on the white vertices.

{\it Case 2: $d \neq 1$ on the black vertices.}

Figure 1 shows that for
$\fs \fu(2,1)$, $\fs \fo(2,q)$ and $\fs \fp(2,\br)$,
we may have $d=1$ on the
white vertices and $d \neq 1$ on the black vertices.
So they are not covered by Case 1.
When $\pa$ has two black vertices, $\fk$ has a
1-dimensional center $\fz$.
Since $d \neq 1$ on the black vertices, we have $\si|_\fz =-1$
\cite[Prop.4.5(b)]{tams}.
But by (\ref{bak}), $\si = \Ad_t$
acts trivially on $\fz$, so this is a contradiction.

We have obtained contradictions in Cases 1 and 2.
Hence $\si \not\in \inn(\fg)$, which proves
Proposition \ref{onto}.
\end{proof}

\sk

The above two propositions deal with finite order automorphisms.
The next proposition says that it indeed suffices
to consider finite order representatives of
$\aut(\fg)/\inn(\fg)$.

\begin{proposition}
Each connected component of $\aut(\fg)$ contains
an element of finite order.
\label{nei}
\end{proposition}
\begin{proof}
Let $\fu = \fk + i \fp$.
Since $\fg$ and $\fu$ are real forms of $\fgc$,
we may regard $\aut(\fg)$ and $\aut(\fu)$
as subgroups of $\aut(\fgc)$ by
$\bc$-linear extensions.
Let ${\mathcal K} = \aut(\fg) \cap \aut(\fu)$.
Then $\aut(\fg)$ has maximally compact subgroup
${\mathcal K}$ with Cartan decomposition
$\aut(\fg) = {\mathcal K}{\mathcal P}$,
where ${\mathcal P}$ is diffeomorphic to a
vector space \cite[Thm.1]{mu}.
Since ${\mathcal K}$ is compact and intersects
every connected component of $\aut(\fg)$,
it follows that every connected component
of $\aut(\fg)$ contains an element of finite order
\cite[Lemma 2.16]{hg}.
\end{proof}

\sk

\noindent {\it Proof of Theorem \ref{main}:}

By Proposition \ref{nei},
it suffices to consider finite order representatives
of $\aut(\fg)/\inn(\fg)$.
They can be represented by markings on $\pa$
\cite[Thm.1.3]{tams}.

We first check that (\ref{them}) is a group homomorphism.
Suppose that $\si'$ and $\si''$ are respectively represented
by markings $(c',d')$ and $(c'',d'')$ with respect to
the root vectors $\{X_\al\}_\pa$, namely
$\si' X_\al = c_{d' \al}' X_{d' \al}$ and
$\si'' X_\al = c_{d'' \al}'' X_{d'' \al}$ by (\ref{repr}).
Let $\si = \si' \si''$. Then
\[ \si X_\al = \si' \si'' X_\al = \si'(c_{d'' \al}'' X_{d'' \al})
= c_{d' d'' \al}' c_{d'' \al}'' X_{d' d'' \al} .\]
So $\si$ is represented by $(c,d)$, where $d = d' d''$ and
$c_\al = c_\al' c_{(d')^{-1}\al}''$. It follows that
\begin{equation}
 \prod_\pa c_\al^{a_\al} =
\prod_\pa (c_\al')^{a_\al} \prod_\pa (c_{(d')^{-1}\al}'')^{a_\al}
= \prod_\pa (c_\al')^{a_\al} \prod_\pa (c_\al'')^{a_\al} .
\label{may}
\end{equation}
The last expression of (\ref{may}) is due to $a_\al = a_{d' \al}$
for all $\al \in \pa$.
By $d = d' d''$ and (\ref{may}), it follows that (\ref{them})
is a group homomorphism.

We shall check that (\ref{them}) is injective and surjective.
Recall that $\thq$ is a Cartan involution of $\fg$,
and $r$ is the order of $\bth \in \aut(\dy)$.
We have $r=1$ if $\mbox{rank } \fg = \mbox{rank } \fk$,
and $r=2$ if $\mbox{rank } \fg > \mbox{rank } \fk$.

Let $\si \in \aut_n(\fg)$ be represented by a marking $(c,d)$ on $\pa$.
By Proposition \ref{satu},
\[
\begin{array}{cl}
r=1: & d=1 \; \Longrightarrow \; \si \in \inn(\fg) ,\\
r=2: & (d, \prod_\pa c_\al^{a_\al}) =(1,1)
\; \Longrightarrow \; \si \in \inn(\fg) .
\end{array}
\]
It implies that the kernel of (\ref{them})
is trivial, so (\ref{them}) is injective.

Next we prove that (\ref{them}) is surjective.
For $r=2$,
if $1 \neq \prod_\pa c_\al^{a_\al} \in \bz_2$,
then the arguments of Proposition \ref{satu}(b) show that
$\si \not\in \inn(\fg)$.
So together with Proposition \ref{onto}, we have
\[
\begin{array}{cl}
r=1: & d \neq 1 \; \Longrightarrow \; \si \not\in \inn(\fg) ,\\
r=2: & (d, \prod_\pa c_\al^{a_\al}) \neq(1,1)
\; \Longrightarrow \; \si \not\in \inn(\fg) .
\end{array}
\]
It implies that (\ref{them}) is surjective.
Hence (\ref{them}) is a group isomorphism.$\hfill$$\Box$

\sk

By checking the list of $\aut(\fg)/\inn(\fg)$ in
\cite[Cor.2.15]{gu} for each $\fg$, we see that it is indeed
isomorphic to our $\aut(\pa) \times \bz_r$.
For example, consider the painted diagram of $\fs \fu(p,q)$
in Figure 1, where $p> 1$ or $q> 1$.
If $p \neq q$, then Figure 4(a) shows that $\aut(\pa) \cong \bz_2$.
If $p = q$, then Figure 4(b) shows that
$\aut(\pa) \cong \bz_2 \times \bz_2$.

\begin{figure}[h]
{\qquad \includegraphics[width=10cm,height=2.2cm]{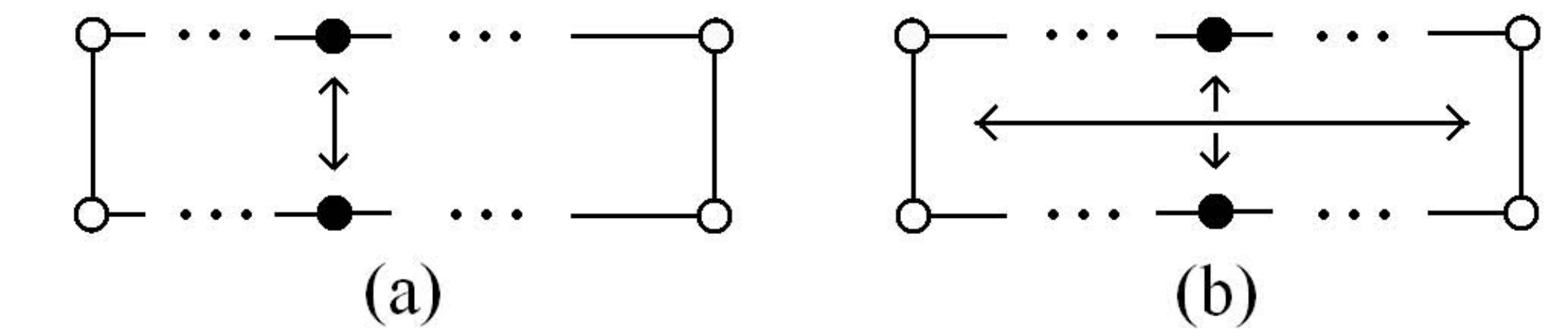}}
\caption{$\aut(\pa)$ for $\fs \fu(p,q)$.}
\end{figure}

Indeed the fourth and fifth items of \cite[Cor.2.15]{gu}
reveal $\aut(\fg)/\inn(\fg)$ as follows,
which verifies that $\aut(\fg)/\inn(\fg) \cong \aut(\pa)$.
\[
\begin{tabular}{|c|c|} \hline
 $\fg$ & $\aut(\fg)/\inn(\fg)$ \\ \hline
 $\fs \fu(p,q) , p \neq q$ & $\bz_2$ \\
 $\fs \fu(p,p)$ & $\bz_2 \times \bz_2$
\\ \hline
\end{tabular}
\]

\sk

\noindent{\it Proof of Corollary \ref{exte}:}

Decompose $\fg$ into
simple ideals $\fc_1 + ...  + \fc_m + \fg_1 + ... + \fg_n$,
where $\fc_i$ is complex and $\fg_i$ is a real form of
a complex simple Lie algebra \cite[Thm.6.94]{kn}.
Let $\dy_i$ denote the Dynkin diagram of $\fc_i$,
and let $\aut(\fc_i)$ denote the $\br$-linear
(not $\bc$-linear) automorphisms on $\fc_i$.
Then \cite[Thm.2.21]{gu}
\begin{equation}
 \aut(\fc_i)/\inn(\fc_i) \cong
\aut(\dy_i) \times \bz_2 .
\label{aam}
\end{equation}
Here $\bz_2$ is generated by a Cartan involution of $\fc_i$,
which is conjugation with respect to a maximally compact
subalgebra.

Next we consider $\fg_i$.
If $\fg_i$ is noncompact, then Theorem \ref{main} gives
\begin{equation}
 \aut(\fg_i)/\inn(\fg_i) \cong
\aut(\pa_i) \times \langle [\thq_i] \rangle .
\label{aan}
\end{equation}
If $\fg_i$ is compact, we let $\pa_i$ be the Dynkin diagram
of $\fg_i^\bc$.
Then (\ref{aan}) holds again \cite[Ch.IX Thm.5.4]{he},
where $[\thq_i] = [1]$ is trivial.

Let $\Ga$ be the group of all permutations
on isomorphic ideals among $\{\fc_i\}$
and $\{\fg_j\}$. Then
\begin{equation}
 \aut(\fg) \cong \prod_1^m \aut(\fc_i) \prod_1^n \aut(\fg_j)
\rtimes \Ga .
\label{aap}
\end{equation}
We can identify $\Ga$ with all permutations on isomorphic diagrams
among $\{\dy_i\}$ and $\{\pa_j\}$
(but do not permute $\dy_i$ with $\pa_j$). Then
\begin{equation}
 \aut(\dy) \times \aut(\pa)
= \prod_1^m \aut(\dy_i) \prod_1^n \aut(\pa_j) \rtimes \Ga .
\label{aaq}
\end{equation}
Corollary \ref{exte} follows from
(\ref{aam}), (\ref{aan}), (\ref{aap}) and (\ref{aaq}).
$\hfill$$\Box$


\end{document}